\documentclass[twoside,a4paper,12pt]{article}
\usepackage{amssymb, amsmath, amsthm}
\usepackage{color}
\usepackage{url}

\newdimen\myMargin
\textwidth \paperwidth
\advance\textwidth -2\myMargin
\textheight \paperheight
\advance\textheight -2\myMargin
\advance\oddsidemargin \myMargin
\advance\evensidemargin \myMargin
\advance\topmargin \myMargin
\advance\textheight -17 true mm
\advance\topmargin 12.5 true mm

\newcommand{\Aalg}{{\cal A}}
\newcommand{\A}{{\cal A}}
\def\B{{\cal B}}

\def\tr{{\mathrm{tr\,}}}
\def\Rad{{ Rad}}
\newcommand{\K}{\mathbb K}

\newcommand{\Q}{\mathbb Q}

\newcommand{\F}{\mathbb F}

\newcommand{\Z}{\mathbb Z}

\newcommand{\ve}[1]{\mathbf{#1}}

\newtheorem{theorem}{Theorem}
\newtheorem{lemma}[theorem]{Lemma}
\newtheorem{corollary}[theorem]{Corollary}
\newtheorem{proposition}[theorem]{Proposition}
\newtheorem{definition}[theorem]{Definition}

\theoremstyle{remark}
\theoremstyle{remark}\newtheorem{remark}[theorem]{Remark}

\title{Computing explicit isomorphisms with full
matrix algebras over $\F_q(x)$}

\author{\normalsize
 \begin{minipage}{0.3\linewidth}
    \large
    G\'abor Ivanyos \\
    \footnotesize
Institute for Computer Science and Control,
Hungarian Acad. Sci. \\
    \texttt{Gabor.Ivanyos@sztaki.mta.hu} \\
    \normalsize
  \end{minipage}
  \qquad
 \begin{minipage}{0.3\linewidth}
    \large
    P\'eter Kutas \\
    \footnotesize
    Central European University, Department of Mathematics and its Applications \\
    \texttt{Kutas\_Peter@phd.ceu.edu} \\
    \normalsize
  \end{minipage}
  \qquad
  \begin{minipage}{0.3\linewidth}
    \large
    Lajos R\'onyai \\
    \footnotesize
     Institute for Computer Science and Control,
     Hungarian Acad. Sci. \\
    Dept. of Algebra, Budapest
    Univ. of Technology and Economics \\
    \texttt{lajos@ilab.sztaki.hu}
    \normalsize
  \end{minipage}
}

\begin{document}

\maketitle

\begin{abstract}
We propose a polynomial time $f$-algorithm
(a deterministic algorithm which uses an oracle for
factoring univariate polynomials over $\F_q$)
for computing an isomorphism (if there is any) of a finite dimensional
 $\F_q(x)$-algebra $\A$  given by structure constants with
the algebra of $n$ by $n$ matrices with entries
from $\F_q(x)$. The method is based on computing
a finite $\F_q$-subalgebra of $\A$ which is the intersection
of a maximal $\F_q[x]$-order and a maximal $R$-order,
where $R$ is the subring of $\F_q(x)$ consisting
of fractions of polynomials with denominator having
degree not less than that of the numerator.
\end{abstract}


\bigskip
\noindent
{\textbf{Mathematics Subject Classification:} 68W30, 16Z05, 16M10

\section{Introduction}

Decomposing finite dimensional associative algebras
over a field $\K$ include the tasks of isolating the
radical, computing simple components of the radical-free
part and finding minimal one-sided ideals within these
simple components. In this paper we consider the case
$\K=\F_q(x)$ where $\F_q$ is the finite field having $q$
elements ($q$ is a prime power). Decomposing algebras over $\F_q(x)$ can be applied
for example to factorization problems in certain skew
polynomial rings, see the work \cite{GZ03} of
Giesbrecht and Zhang and the recent paper \cite{GTLN}
by G\'omez-Torrecillas, Lobillo and Navarro. The
first two tasks mentioned above can be accomplished
by the polynomial time f-algorithm proposed in the
work of the first and the third authors with Sz\'ant\'o~\cite{IRSz}.
The third problem, finding minimal one-sided ideals in
simple algebras appears to be more difficult. In this paper
we propose a solution which works in the special case when
the algebra happens to be isomorphic to the full matrix algebra
$M_n(\F_q(x))$.

If $\A$ is a $\K$-algebra isomorphic to $M_n(\K)$, then
finding a minimal left ideal (or, more generally, finding
an irreducible $\A$-module) is equivalent to constructing
an isomorphism $\phi:\A\rightarrow M_n(\K)$. Indeed,
the matrices having possibly nonzero entries only in the first
column form a minimal left ideal in $M_n(\K)$, so the inverse image
under $\phi$ is a minimal left ideal in $\A$. Conversely, if
$M$ is an irreducible (that is, an $n$-dimensional)
$\A$-module,  then the action of $\A$ on $M$ gives an isomorphism
$\A\cong M_n(\K)$. Therefore the task of finding a minimal left ideal
is also known as the
{\em explicit isomorphism problem}.


\medskip
Recall, that for an algebra ${\A}$ over a field $\K$
and a $\K$-basis   $a_1, \ldots ,a_m$ of ${\A}$ over $\K$
the products $a_ia_j$ can be expressed as
linear combinations of the $a_i$:
$$ a_ia_j=\gamma _{ij1}a_1+\gamma _{ij2}a_2+\cdots +\gamma _{ijm}a_m. $$
The elements $\gamma _{ijk}\in \K$ are called structure constants. In
this paper an algebra is considered to be given as a collection of
structure constants.

Here we consider the explicit isomorphism problem for
$\K=\mathbb{F}_q(x)$.
For the case $\K=\F_q$ the polynomial time
f-algorithm given in \cite{Ro1}
by the third author gives a solution.
See also \cite{IKRS} for related deterministic
methods.
Recently the first and the third authors with
Schicho \cite{IRS} found an algorithm
for solving the explicit isomorphism problem in the case of number fields.
Their algorithm is a polynomial time ff-algorithm
(it is allowed to call oracles for factoring polynomials over
finite fields and for factoring integers), assuming that the
degree of the matrix algebra and the degree of the number field
over $\mathbb{Q}$ are bounded. They combined algebraic
techniques with tools from lattice geometry. Some improvements
were given in \cite{ILR}. Their results have various applications, for instance in arithmetic geometry (see \cite{Cremona1}, \cite{Cremona2}, \cite{Cremona3}).



The structure of the paper will be the following. First we develop the
necessary notions concerning polynomial lattices.
In Section~\ref{sec:lattices} we summarize the main tools
for handling lattices over $\F_q[x]$. The orthogonality defect
inequality and the basis reduction algorithm of Lenstra~\cite{Lenstra}
are discussed here. We shall also use extensions by Paulus~\cite{Paulus}.

In the next section we state and prove certain facts about maximal orders
over polynomial rings. Then we use them to construct
maximal $\F_q[x]$ and $\F_q[\frac{1}{x}]$ orders in $\A$.
The algorithms run in polynomial time if one is allowed to
call oracles for factoring univariate polynomials over finite fields
(it is a polynomial f-algorithm).

Let $R$ be the subring of $\mathbb{F}_q(x)$ consisting of those rational
functions where the degree of the
denominator is at least as high as the degree of the numerator.
The main structural result of the paper is that the intersection of a maximal $R$-order and a
maximal $\mathbb{F}_q[x]$-order is a finite dimensional
$\mathbb{F}_q$-algebra which contains a primitive idempotent of $\A$.
This theorem and the resulting algorithms are described in
Section~\ref{sec:findidemp}:
we propose an algorithm to find a primitive idempotent of $\A$.
Finally we arrive at the following theorem:

\begin{theorem}
\label{thm:Main}
Let $\A$ be isomorphic to $M_n(\mathbb{F}_q(x))$, and
given by structure constants.
Then there exists a polynomial
(in $n$ and in the size of the structure constants)
f-algorithm which finds an explicit isomorphism
between $\A$ and $M_n(\F_q(x))$.
\end{theorem}

Together with the polynomial time randomized algorithms
of Cantor and Zassenhaus~\cite{CZ} (or, when $q$ is a power of
a prime bounded by a constant, with
the deterministic method of Berlekamp~\cite{Berlekamp}), this gives
a randomized polynomial time solution in general
(and a deterministic polynomial time algorithm for small
characteristic)
for the explicit isomorphism problem
in the special case $\K=\F_q(x)$. We remark that the main ideas
of this paper can be extended to the case $\K=\F(x)$ for various
fields $\F$ of constants. (One just needs efficient methods
for decomposing finite dimensional algebras over $\F$, and
lattice basis reduction over $\F [x]$.) However,
extending our algorithms to finding minimal left ideals in algebras
which are isomorphic to full matrix algebras over finite
{\em extensions} of $\F_q(x)$ looks more difficult.

Our main aim was in this work to show the existence of a polynomial time
f-algorithm for the explicit isomorphism problem over $\F_ q(x)$.
No attempt has been made to optimize
exponents and implied constants in the time bound. Those would require
substantial further work. Our approach, in return, allowed a relatively short
description of the methods and arguments.

\section{Lattices over function fields}
\label{sec:lattices}

Most of our definitions and lemmas come from the seminal
paper~\cite{Lenstra} of A.~K.~Lenstra.
He introduced the notion of reduced basis and found an
algorithm which finds a shortest vector
in polynomial time in lattices over $\F_q[x]$
(he considered sublattices of $\mathbb{F}_q[x]^m$).
Note that the analogous problem is NP-hard in the case of integer
lattices~\cite{Ajtai}.
First we state certain definitions about $\mathbb{F}_q[x]$-lattices
in $\F_q(x)^m$.

\begin{definition}
Let $f,g\in \mathbb{F}_q[x]$.
Then we set $|\frac{f}{g}|=\mathrm{deg} (f)-\mathrm{deg} (g) $.
We will refer to $|.|$ as the {\em valuation} (or {\em degree}) of
an element of $\F_q(x)$. We set $|0|=-\infty$.
Let $\ve v=(v_1,\dots,v_m)^T\in \mathbb{F}_q(x)^m$.
Then the {\em valuation} (or {\em degree}) of the vector $\ve v$ is $|\ve
v|=\max(|v_1|,\dots,|v_m|)$.
\end{definition}

\begin{definition}
$L$ is a {\em full
lattice} in $\F_q(x)^m$
if $L=\{\alpha_1\ve b_1+\dots+\alpha_m\ve b_m|~ \alpha_i\in \mathbb{F}_q[x]\}$
where $\ve b_1,\dots, \ve b_m$
is a basis (over $\mathbb{F}_q(x)$) in $\mathbb{F}_q(x)^m$.
\end{definition}
\begin{definition}
Let $\ve b_1, \ve b_2, \dots, \ve b_m\in \mathbb{F}_q(x)^m$. Then
 the {\em orthogonality defect} $OD(\ve b_1,\dots, \ve b_m)$
is defined as $OD(\ve b_1,\dots, \ve b_m)=\sum_{i=1}^{m} |\ve
b_i|-|\det(B)|$ where $B$ is the matrix whose
columns are the $\ve b_i$, ($i=1,\dots,m$).
\end{definition}

The following lemma is from \cite{Lenstra}.
However, there it is stated in a slightly weaker form than we need it in
this paper. So we state and prove the lemma here as well. The proof is also from \cite{Lenstra}.
\begin{lemma}\label{orthogonality defect}
Let $\ve b_1, \ve b_2, \dots, \ve b_m\in \mathbb{F}_q(x)^m$ be linearly independent and $\ve
a=\sum_{i=1}^{m} \alpha_i\ve b_i $ where $\alpha_i\in \mathbb{F}_q[x]$. Then the following holds for every $i$:
\begin{equation}
|\alpha_i|\leq |\ve a|+OD(\ve b_1,\dots, \ve b_m)-|\ve b_i|
\end{equation}
\end{lemma}
\begin{proof}

Consider the $\alpha_i$ as unknowns. In this case we have $m$
linear equations and $m$ variables so we can use Cramer's rule.
Note that $\{\ve b_i\}_{i=1}^m$ is a basis so the determinant of the
coefficient matrix $B$ is non-zero. By Cramer's rule
$\alpha_i$ is equal to the quotient of 2 determinants. In other words
$\alpha_i$ multiplied by the determinant of the lattice is
equal to the determinant where the $i$th column of $B$ is switched to
$\ve a$. Since these two sides are equal, their valuations are equal also
(on both sides we have elements from $\mathbb{F}_q(x)$).
Note that the valuation of a determinant can be bounded from above by the
sum of the valuations of its columns. To formalize this last sentence:
\begin{eqnarray*}
|\alpha_i|+|\det(B)|
&\leq& |\ve b_1|+|\ve b_2|+\dots +|\ve b_{i-1}|+|\ve b_{i+1}|+\dots +|\ve
b_m|+|\ve a|\\
&=& \sum_{i=1}^{m} |\ve b_i|-|\ve b_i|+|\ve a|.
\end{eqnarray*}
After rearranging we obtain the result.
\end{proof}
 An implication of this lemma is the following. If we have a vector with
small valuation, then the coefficients corresponding to a basis are
also small, if the orthogonality defect of the basis is small.
This also suggests that an ideal basis is one whose orthogonality defect is 0. This motivates the following definition.
\begin{definition}
A basis $\ve b_1, \ve b_2, \dots, \ve b_m\in \mathbb{F}_q(x)^m$ is called
{\em reduced} if $OD(\ve b_1,\dots, \ve b_m)=0$.
\end{definition}

Lenstra proposed a polynomial time method~\cite[Algorithm 1.7]{Lenstra}
to compute reduced bases of sublattices of $\F_q[x]^m$. We quote
\cite[Proposition 1.14]{Lenstra} below.

\begin{proposition}
Let $\ve b_1,\ve b_2,\dots, \ve b_m$ be over $\F_q(x)$ linearly independent
vectors from $\F_q[x]^m$ and let $L$ be the $\F_q[x]$-lattice they
generate. Let $M=\max_{1\leq i\leq m}(|\ve b_i|)$ and let $M'=\max(M,1)$.
Then there exists an algorithm which takes
$O(m^3M'(OD(\ve b_1,\dots, \ve b_m)+1))$
arithmetic operations in $\mathbb{F}_q$
and returns a reduced basis $\ve c_1,\dots,\ve c_m$ of $L$
for which we have
$|\ve c_i|\leq M$ ($i=1,\ldots,m$).
\end{proposition}

This result can
be extended
to find a reduced basis of a full lattice in $\mathbb{F}_q(x)^m$.
Let us assume that we have a basis $\ve b_1,\ve b_2,\dots,\ve b_m$
in $\mathbb{F}_q(x)^m$.
Let $L$ be the $\mathbb{F}_q[x]$-lattice generated by these vectors and
let $B$ be the matrix with columns $\ve b_1,\dots,\ve b_m$. Let $\gamma$ be
the least common multiple of all the denominators of the entries of $B$.
We consider the lattice $L'$ generated by
$\gamma\ve b_1,\dots, \gamma\ve b_m$.
Note that $L'\in \mathbb{F}_q[x]^m$.
So using Lenstra's algorithm one can find a
reduced basis $\ve c_1,\dots,\ve c_m$ in $L'$.
Note that $|\det L'|=|\det L|+m|\gamma|$. This implies that choosing
$\ve b_i'=\frac{1}{\gamma}\ve c_i$ we get a reduced basis of $L$.
Since the orthogonality defect of $\ve b_1,\dots ,\ve b_m$
is the same as the orthogonality defect of
$\gamma\ve b_1,\dots ,\gamma\ve b_m$, we obtain the following:

\begin{proposition} \label{Lattice reduction}
 Let $\ve b_1,\ve b_2,\dots, \ve b_m$ be a basis in
$\mathbb{F}_q(x)^m$ and let $L$ be the
$\mathbb{F}_q[x]$-lattice they generate. Let
$\gamma$ be the least common multiple of all the
denominators for the entries of
$\ve b_1,\ve b_2,\dots, \ve b_m$.
Let $M=|\gamma|+\max_{1\leq i\leq m}(|\ve b_i|)$
and let $M'=\max(M,1)$.
Then there exists an algorithm which takes
$O(m^3M'(OD(\ve b_1,\dots, \ve b_m)+1))$
arithmetic operations in $\mathbb{F}_q$
and returns a reduced basis of $L$.
\end{proposition}
\begin{proof}
Lenstra's method~\cite[Algorithm 1.7]{Lenstra} together with
its analysis~\cite[Proposition 1.14]{Lenstra} gives the result.
\end{proof}

Given an integer $k$, the set of elements of the lattice
whose valuation is smaller than $k$ is a finite
dimensional $\mathbb{F}_q$-vector space (this is a
consequence of Lemma \ref{orthogonality defect}), and a basis of
this vector space can also be computed efficiently.

The algorithm
of Proposition~\ref{Lattice reduction}
 finds a reduced basis of a lattice which is given by a basis.
However, one can ask the following question: what happens if the lattice is only given
by an $\F_q[x]$-module generating system?
In such situations an algorithm by Paulus~\cite[Algorithm 3.1.]{Paulus}
is applicable. It finds a reduced basis of a lattice
in $\F_q(x)^m$ given by a system of generators. We shall make use of
the fact that the valuations of the reduced basis obtained by Paulus'
algorithm will not be greater than those of the given generators.

\section{Maximal orders over $\mathbb{F}_q[x]$}
\label{sec:maxord}

\subsection{Preliminaries}

In this subsection we assume that $R$ is a principal ideal domain
with quotient field $\K$ and $\Aalg$ is a central simple algebra
isomorphic to $M_n(\K)$. Recall that an {\em $R$-order} in $\Aalg$
is a full $R$-lattice which is at the same time a subring
of $\Aalg$ containing the identity element.
Maximal orders are orders maximal with
respect to inclusion. We start with rephrasing
\cite[Theorem 21.6]{Reiner}, specialized to this setting.

\begin{proposition}
\label{prop:maxord_pid}
Let $\A=\mathrm{Hom}_\K(V,V)$ where $V$ is a vector space of
dimension $n$ over $\K$. Let $L$ be any full $R$-lattice in $V$.
Then $\mathrm{Hom}_{R}(L,L)$, identified with the subring
$$O(L)=\{a\in \A: aL\leq L\}$$
of $\A$, is a maximal $R$-order in $\A$, and all maximal
orders are of this form.
\end{proposition}

In terms of matrices, the second statement of the
theorem gives the following.

\begin{corollary}
\label{cor:max_orders}
Assume that $\Lambda$ is a maximal $R$-order
in $M_n(\K)$.
Then there exists an invertible matrix
$P\in M_n(\K)$ such that
$\Lambda=PM_n(R)P^{-1}$.
\end{corollary}
\begin{proof}
The theorem with with $V=\K^n$ gives that
every maximal $R$-order in $M_n(\K)$
is ${\cal O}(L)$ for a full $R$-lattice
$L$ in $\K^n$. Let $P$ be a matrix
whose columns are an $R$-basis of $L$.
\end{proof}

We remark that this claim can be found for quaternion algebras in \cite[Exercise 4.2]{Vigneras}.

\smallskip

Our eventual aim is to construct a maximal $R$-order
in $M_n(\K)$. We will construct an initial order $\Lambda_0$
in a rather straightforward way and iteratively enlarge
it. Strictly speaking, our initial object $\Lambda_0$
will not be an order. We say that an $R$-subalgebra $\Lambda$
of $\A$ is an {\em almost $R$-order} in $\A$ if
it is a full $R$-lattice in $\A$. Thus orders
are almost orders containing the identity element of $\A$.
It turns out that if $\Lambda_0$ is an almost $R$-order,
then the $R$-lattice generated by $\Lambda_0$ and the identity
element of $\A$ is an $R$-order.

Discriminants enable us to control the depth of
chains of (almost) orders and will also be useful in representing
orders efficiently. The {\em reduced trace}, $\tr a$, of an
element $a$ of an $\A$ is simply the trace of $a$ as an
$n$ by $n$ matrix. (This is well defined by
the Noether-Skolem theorem.) To compute
reduced traces efficiently, it is not necessary to know
an isomorphism $\A\cong M_n(\K)$. If $n$ is not
divisible by the characteristic of $\K$, then $\tr a$ is
$\frac{1}{n}$ times the trace of the image of $a$ under
the regular representation of $a$. In general, the reduced
trace can be computed by taking an appropriate coefficient
of the $n$th root of the characteristic polynomial of
the regular representation. This is because the regular representation
of $\A$ decomposes as a direct sum of $n$ copies of the
standard $n$-dimensional (irreducible) representation.

The {\em bilinear trace
form} on $\Aalg$ is the symmetric bilinear function
$(a,b)\mapsto \tr ab$. As the matrix corresponding to an
element of an almost $R$-order $\Lambda$
is similar to a matrix with entries of $R$,
the reduced trace of any element of $\Lambda$ is
from $R$. The discriminant
$d(\Lambda)$ can be defined as the principal ideal of $R$ generated by
the determinant of the Gram matrix
$\left( \tr b_ib_j\right )_{i,j=1}^{n^2}$ where $b_1,\ldots,b_{n^2}$
are an $R$-basis for $\Lambda$. It is nonzero
and independent of
the choice of the basis. We can loosely think of $d(\Lambda)$
as an element of $R$, defined up to a unit of $R$. As the
bilinear trace form is non-degenerate, we have the following
(see \cite[Exercise 10.3]{Reiner}).

\begin{proposition}
\label{prop:disc-incr}
Let $\Lambda$ and $\Gamma$ be almost $R$-orders in $\A$ such that
$\Lambda\subseteq\Gamma$. Then $d(\Gamma) | d(\Lambda)$
and $\Lambda=\Gamma$ if and only if $d(\Gamma)=d(\Lambda)$.
\end{proposition}

The following statement gives an $R$-lattice as an upper
bound for $R$-orders containing a given almost order.
An extension to more general rings $R$ is used in the proof of
\cite[Theorem~10.3]{Reiner}. For orders over principal ideal domains it is
stated explicitly in~\cite[Proposition~2.2]{RLIR}. As we need
a slight generalization to almost orders, we give a proof for completeness.

\begin{proposition}
\label{prop:discriminant}
Let $\Lambda$ and $\Gamma$ be almost $R$-orders in $\A$ such that
$\Lambda\subseteq\Gamma$.
Then $\Gamma \subseteq \frac{1}{d}\Lambda$ where $d=d(\Lambda)$.
\end{proposition}

\begin{proof}
Let $b_1,\ldots,b_{n^2}$ be an $R$-basis for $\Lambda$. Then
an element $a\in \Gamma$ can be written as $a=\sum_{i=1}^{n^2}
\alpha_ib_i$ with $\alpha_i\in \K$ ($i=1,\ldots,n^2$).
For $j=1,\ldots,{n^2}$ put $\beta_j=\tr ab_j$.
Then the elements $\beta_j$ are in $R$ because the elements $ab_j$ are in
the almost order $\Gamma$ which is contained in an $R$-order
and hence have reduced trace from $R$. By linearity, we
have $\sum_{i} \alpha_i\tr b_ib_j=\beta_j$. Cramer's rule gives
that each $\alpha_i$ is a quotient of an element of $R$ and
$d$, which means that $a\in \frac{1}{d}\Lambda$.
\end{proof}

An algorithmic consequence is that it is possible to represent $R$-orders
containing a given almost order $\Lambda$ as submodules of the factor
module $\frac{1}{d}\Lambda/\Lambda$. This will be particularly useful
when $R=\F_q[x]$, in which case this factor is an $n^2\deg d$-dimensional
vector space over $\F_q$.

Our algorithm for computing maximal orders
is an adaptation of the method proposed by
the first and third authors for the case $R=\Z$
in~\cite{RLIR}. The method is discussed in the context of
global fields in the Ph.~D.~thesis of the first
author~\cite{Ithesis}. The algorithm finds a maximal
order in a separable algebra over a global field. The algorithm
proposed in this paper works also for separable algebras in a similar
fashion. However, we only consider the case of full matrix algebras as 
some minor details (e.g. those regarding the trace form) are simpler in this case (and this is the only case we need later on). 

For completeness, we include proofs of
statements that are not rigorously proved for
general principal ideal rings in \cite{RLIR}.

Let $M$ be a full $R$-lattice in $\A$. Then the left order of $M$ is defined by
\begin{equation*}
O_l(M)=\{a\in \A|\; aM\subseteq M\}.
\end{equation*}
The set $O_l(M)$ is known to be an $R$-order of $\A$, see \cite[Chapter 8]{Reiner}.
It actually follows from the fact that $O_l(M)$ is isomorphic to
the intersection of two $R$-algebras: the
image of $\A$ under the left regular representation
and $\mathrm{Hom}_R(M,M)$ (embedded into $\mathrm{Hom}_\K(\A,\A)$).

The next two lemmas will be important tools for the algorithm
which finds maximal orders. The first one reduces the question
of enlarging an order over $R$ to a similar task for
$R_\pi$-orders where $\pi$ is a prime element of $R$.
Here $R_\pi\leq \K$ denotes the localization of $R$ at
the prime ideal $R\pi$,
that is, $R_\pi=\{\frac{\alpha}{\beta}:\alpha,\beta\in
R\mbox{~with~}\pi\not|\beta\}$.
If $\Gamma$ is an $R$-order in $\A$, then $\Gamma_\pi=R_\pi\Gamma$
is an $R_\pi$-order.

\begin{lemma}\label{lem:enlarging}
Let $\pi$ be a prime element of $R$ and $\Gamma$ be an
$R$-order in $\A$.
Suppose that $J$ is an ideal of
$\Gamma_\pi$ such that $J\geq \pi\Gamma_\pi$ and $O_l(J)>\Gamma_\pi$.
Put $I=\Gamma\cap J$. Then we have $I\geq \pi\Gamma$ and $O_l(I)>\Gamma$.
\end{lemma}

This lemma is stated for $R=\Z$ in  \cite[Lemma~2.7 ]{RLIR}.
The proof goes through for any principal domain $R$. We include it
for completeness.

\begin{proof}
Clearly $I\geq \pi\Gamma$ and $I$ is an ideal of $\Gamma$.
We also have $J=R_\pi I$. Let $a\in O_l(J)\backslash\Gamma_\pi$.
 Let $a_1,a_2,\dots,a_t$ be a generating set of
$I$ as an $R$-module. Then these elements generate $J$ as an $R_\pi$-module whence
 for $i=1,\dots,t$ we have
\begin{equation}
aa_i=\frac{\alpha_{i1}}{\beta_{i1}}a_1+\ldots+\frac{\alpha_{it}}{\beta_{it}}a_t,
\end{equation}
where $\alpha_{ij},\beta_{ij}\in R$ and $\pi$
does not divide $\beta_{ij}$. Now put
$\beta=\prod_{i,j}\beta_{ij}$. Then $\beta aa_i$ is
in $I$ ($i=1,\ldots,t$), whence $\beta aI\leq I$ and consequently
$\beta a\in O_l(I)$.
Finally we observe that $\beta a$ is not in $\Gamma$ since
$\beta$ is not divisible by $\pi$ and therefore $\beta a\in \Gamma$
would imply $a\in \Gamma_\pi$. The proof is complete.
\end{proof}

The next simple statement is stated in \cite[Proposition~2.8]{RLIR}
for $R=\Z$. It enables us to use $\Lambda$ in place of $\Lambda_\pi$
in computations regarding sufficiently large one or
two-sided ideals of $\Lambda_\pi$.

\begin{proposition}\label{prop:local}

Let $\Lambda$ be an $R$-order in $\A$ and $\pi$ be a prime of $R$.
Then the map $\Phi: x\mapsto
x+\pi\Lambda_\pi (x\in\Lambda)$ induces an isomorphism of rings
$\Lambda/\pi\Lambda\cong \Lambda_\pi/\pi\Lambda_\pi$.
\end{proposition}
\begin{proof}

Clearly $\Phi: \Lambda\rightarrow \Lambda_\pi/\pi\Lambda_\pi$
is an epimorphism of rings. It is straightforward to check that
its kernel is $\pi\Lambda$.
\end{proof}

Now we quote some further theorems and definitions from \cite{RLIR}.
The next statement is \cite[Proposition 3.1]{RLIR}.

\begin{proposition}\label{prop:local2}
Let $\Lambda_\pi$ be an $R_\pi$-order in $\A$.
Then the residue class
ring $\overline{\Lambda}_\pi=\Lambda_\pi/\pi\Lambda_\pi$
is an algebra with identity element over the residue
class field ${\overline R}_\pi=R_\pi/\pi R_\pi\cong R/\pi R$
and $\mathrm{dim}_{\K}\A=\mathrm{dim}_{{\overline R}_\pi}
\overline{\Lambda}_\pi$. If $\Phi:\Lambda_\pi\rightarrow\overline{\Lambda}_\pi$
is the canonical epimorphism, then $\pi\Lambda_\pi\subseteq Rad(\Lambda_\pi)=
\Phi^{-1}Rad(\overline{\Lambda}_\pi)$ and $\Phi$ induces a
ring isomorphism $\Lambda_\pi/Rad(\Lambda_\pi)\cong
\overline{\Lambda}_\pi/Rad(\overline{\Lambda}_\pi)$.
\end{proposition}

Now we will introduce the important concept of extremal orders:

\begin{definition}
Let $\Lambda_\pi$ and $\Gamma_\pi$ be $R_\pi$-orders in $\A$.
We say that $\Gamma_\pi$ {\em radically contains}
$\Lambda_\pi$ if and only if $\Gamma_\pi\supseteq \Lambda_\pi$ and
$Rad(\Gamma_\pi)\supseteq Rad(\Lambda_\pi)$.
This is a partial ordering on the set of $R_\pi$-orders.
Orders maximal with respect to this partial ordering are
called {\em extremal}.
\end{definition}

The next statement is \cite[Proposition 4.1]{RLIR}.

\begin{proposition}\label{prop:extremal1}
An $R_\pi$-order $\Lambda_\pi$ of $\A$ is extremal if and only if
$\Lambda_\pi=O_l(Rad(\Lambda_\pi))$.
\end{proposition}

Finally, we quote \cite[Proposition 4.5]{RLIR}.

\begin{proposition}\label{prop:extremal2}

Let $\Lambda_\pi\subset \Gamma_\pi$ be $R_\pi$-orders in $\A$.
Suppose that $\Lambda_\pi$ is extremal and $\Gamma_\pi$ is
minimal among the $R_\pi$-orders properly containing
$\Lambda_\pi$. Then there exists a two-sided ideal $I$ of $\Lambda_\pi$
minimal among those containing $Rad(\Lambda_\pi)$ such that
$O_l(I)\supseteq \Gamma_\pi$
\end{proposition}

\subsection{The algorithm}

We start 
with a high-level description of the algorithm over
a general principal ideal domain $R$.
Let $R$ be a principal ideal domain, $\K$ its field of fractions.
Suppose that an algebra $\Aalg$, isomorphic to $M_n(\K)$ is given by
structure constants $\gamma_{ij}^k$ ($i,j,k=1,\ldots, n^2$) from $\K$
with respect to a basis $a_1,\ldots,a_{n^2}$.
We assume that these structure constants
are represented as fractions of pairs of elements from $R$.
Let $\delta$ be a common multiple (e.g., the product or the l.~c.~m.)
of the denominators. Then $a_i'=\delta a_i$ ($i=1,\ldots,n^2$)
will be a basis
with structure constants $\delta \gamma_{ij}^k\in R$. Therefore the
$R$-submodule $\Lambda_0$ of $\A$ with basis $a_1',\ldots,a_{n^2}'$
is an almost $R$-order.

We shall compute the discriminant $d=d(\Lambda_0)$. Let
$S=\{\pi_1,\ldots,\pi_r\}$ be the set of the prime
factors of $d$. Observe, that the discriminant of any $R$-order
conjugate to $M_n(R)$ is $1$. This also holds for $R_\pi$-orders
for any prime element $\pi$.
Therefore, by Corollary~\ref{cor:max_orders}
and by Proposition~\ref{prop:disc-incr}, ${\Lambda_0}_\pi$
is a maximal $R_\pi$-order for any prime $\pi$ not in $S$.

Starting with the order $\Lambda$ obtained by taking the
$R$-module generated by $\Lambda_0$ and the identity element,
for each prime in $S$ we test
constructively whether $\Lambda_\pi$ is a maximal $R_\pi$ order
using the two tests described below. By constructiveness we mean
that in the "no" case we construct an $R$-order
$\Gamma\supsetneq \Lambda$. If any of the tests finds such a $\Gamma$,
then we proceed with $\Gamma$ in place of $\Lambda$. Otherwise, if
$\Lambda_\pi$ passes the tests for every $\pi\in S$ then
we conclude that $\Lambda$ is already maximal. By
Proposition~\ref{prop:disc-incr} the number of such rounds is at most
the number of the prime divisors of $d$, counted with multiplicities.

The first test is used to constructively decide whether $\Lambda_\pi$ is
an extremal $R_\pi$-order by checking if $O_l(Rad(\Lambda_\pi))=\Lambda_\pi$
(Proposition \ref{prop:extremal1}).
To this end, we compute the ideal $I=Rad(\Lambda_\pi)\cap \Lambda$.
By Lemma \ref{lem:enlarging}, $\Lambda$ passes the test if and only if
$O_l(I)=\Lambda$. Otherwise $\Gamma=O_l(I)$ is an order strictly
containing $\Lambda$. To compute $I$, we work with the
$n^2$-dimensional $R/\pi R$-algebra $\B=\Lambda/\pi \Lambda$.
From Propositions \ref{prop:local} and \ref{prop:local2} we infer that
$I$ is the inverse image of $Rad(\B)$ with respect to the canonical map
$\Lambda \rightarrow \B$.

If $\Lambda_\pi$ passes the first test, then we proceed with
the test of Proposition \ref{prop:extremal2}: if there exists an ideal $J$
of $\Lambda_\pi$
minimal among the two-sided ideals properly containing $Rad(\Lambda_\pi)$
such that $O_l(J)>\Lambda_\pi$, then we construct an
$R$-order $\Gamma$ that properly contains $\Lambda$.
Like for the first test, we can work in the $R/\pi R$-algebra
$\B=\Lambda/\pi\Lambda$. Let
$J_1,\dots, J_m$ denote the minimal two-sided ideals of
$\B$ which contain $Rad(\B)$. We have $m\leq n^2$. Let $I_i$ denote
the inverse image of $J_i$ with  respect to the
map $\Lambda\rightarrow \B$.
As in the first case we obtain, that we have to compute
the rings $O_l(I_i)$ for $i=1,\dots, m$. We can stop when
$\Lambda<O_l(I_i)$ is detected, because then we have an order
properly containing $\Lambda$.

\subsection{The case $R=\F_q[x]$}

We continue with details of the key ingredients of
an efficient algorithm
for $R=\F_q[x]$
following the lines above.
These will give an $f$-algorithm
whose running time is polynomial in the size of
the input. The input is an array of $n^6$
structure constants represented as fractions
of polynomials.
We assume that the numerators
are of degree at most $d_N$ and the denominators are of
degree at most $d_D$. Thus the size of the input
is around $n^6(d_D+d_N)\log q$.

The l.~c.~m.~of the denominators and hence a basis for
the initial order $\Lambda_0$ can be computed in polynomial time.
The degree of this common denominator is at most $n^6d_D$,
whence $\Lambda_0$ will have a basis $a_1',\ldots,a_{n^2}'$,
where each $a_j'$ is $a_j$, multiplied by a polynomial of degree
at most $n^6d_D$. The structure constants for the basis
$a_1',\ldots,a_{n^2}'$
are polynomials of degree at most $n^6d_D+d_N$.
The discriminant $d=d(\Lambda_0)$ can be
efficiently computed in a direct way following the definition.
The entries of the matrices for the images of $a_j'$ at the regular
representation, written in terms of the basis $a_1',\ldots,a_{n^2}'$
are just structure constants for the basis $a_1',\ldots,a_{n^2}'$.
Therefore these entries are polynomials of degree bounded by $n^6d_D+d_N$
and hence the entries of the Gram matrix of the bilinear trace form are
polynomials of degree
$2n^6d_D+2d_N$. To compute $d(\Lambda_0)$, let $n=p^rk$ where $p$ is the characteristic
of $\F_q$ and $k$ is relatively prime to $p$. Then the characteristic polynomial
of $a_i'a_j'$, is the $n$th power of the characteristic polynomial of
$a_i'a_j'$ as an
$n$ by $n$ matrix. Therefore it is of the form
\begin{eqnarray*}
(X^n-(\tr a_i'a_j')X^{n-1}+\ldots)^n&=&
(X^{nk}-k(\tr a_i'a_j')X^{nk-1}+\ldots)^{p^r}\\
&=& X^{n^2}-
(k\tr a_i'a_j')^{p^r}X^{n^2-p^r}+\ldots.
\end{eqnarray*}
It follows that $d(\Lambda_0)$ is a polynomial
$D_0$ of degree at most $2n^8d_D+2n^2d_N$.

By Proposition~\ref{prop:discriminant}, we have
$\Lambda\leq \frac{1}{D_0}\Lambda_0$
for any $\F_q[x]$-order
$\Lambda\geq \Lambda_0$. Therefore we can represent
$\Lambda$ as the $\F_q[x]$-submodule $\Lambda/\Lambda_0$
of the factor module $\frac{1}{D_0}\Lambda_0/\Lambda_0$.
This factor module is an $n^2\deg D_0$-dimensional
vector space over the field $\F_q$. In fact,
the elements $\frac{x^k}{D_0}a_i'+\Lambda_0$ ($i=1,\ldots,n^2$,
$k=0,\ldots,\deg D_0-1$ form an $\F_q$-basis) and we
represent $\Lambda/\Lambda_0$
by an $\F_q$-basis written in terms of this basis.
Notice that the ideals $I$ whose left order $O_l(I)$
we compute throughout the algorithm are all (left)
$\Lambda_0$-submodules of $\frac{1}{D_0}\Lambda_0$
containing $D_0\Lambda_0$. Observe next that the multiplication
of $\A$ induces an $\F_q$-bilinear map $\mu$ from
$\frac{1}{D_0}\Lambda_0/\Lambda_0\times I/D_0I$
to $\frac{1}{D_0}I/I$. For $a\in \frac{1}{D_0}\Lambda$ and $b\in I$,
one can set
$$\mu(a+\Lambda_0,b+D_0I)=ab+I.$$
This is well defined as
$(\frac{1}{D_0}\Lambda_0)(D_0I)=\Lambda_0I\subseteq I$.
Taking an $\F_q$-basis $b_1,\ldots,b_t$
of $I/D_0I$, the factor $O_l(I)/\Lambda_0$ can be computed
as the intersection of the kernels of the linear maps
$\mu(\cdot,b_i)$ ($i=1,\ldots,t$). As the dimensions
are bounded by polynomials in $n$ and in the degree of $D_0$,
for every $I$ possibly
occurring in
the algorithm, $O_l(I)$ is computable in polynomial time.
Given an intermediate order $\Lambda$, we can compute
the candidate ideals $I$ by computing the radical of $\B=\Lambda/g\Lambda$
for the irreducible factors $g$ of $D_0$ and the minimal
two-sided ideals of $\B$ containing
the radical and finally by taking inverse images of these
at the map $\Lambda\rightarrow \B$.
As $\B$ is an $n^2\deg g$-dimensional
vector space over $\F_q$, its radical and the minimal
two-sided ideals containing it can be
computed in time polynomial in the input size
using for example the deterministic method of the third author
\cite{Ro1}. The minimal two-sided ideals containing the radical,
that is, the simple components of $\B/Rad(\B)$
can be found by the deterministic $f$-algorithm of Friedl and the
third author \cite{FR}.

For $\alpha_{ik}\in\F_q$ ($i=1,\ldots,n^2$, $k=0,\ldots\deg D_0-1$),
the combination
$\sum_{i=1}^{n^2}\sum_{k=0}^{\deg D_0-1}\alpha_{ik}
\frac{x^k}{D_0}a_i'$ of $a_1',\ldots,a_{n^2}'$
has coefficients whose numerators and denominators are polynomials
of degree at most
$\deg D_0\leq 2n^8d_D+2n^2d_N$. Together
with $a_1',\ldots,a_{n^2}'$, such representatives for an $\F_q$-basis
of $\Lambda/\Lambda_0$ give a system of generators over $\F_q[x]$
for $\Lambda$. When $\Lambda$ turns out to be maximal, then
we can use the lattice reduction algorithm by
Paulus~\cite{Paulus}
to obtain a basis for $\Lambda$ consisting of combinations
of $a_1',\ldots,a_{n^2}'$ with coefficients having numerators
and denominators also of degree at most $\deg D_0\leq 2n^8d_D+2n^2d_N$.
(Here we make use of the nature of the reduction algorithm:
it never increases the maximum degree of the coordinates
of the intermediate generators.)
This gives us the following theorem:

\begin{theorem}
\label{thm:maxord-construct}
Let $\A$ be isomorphic to $M_n(\mathbb{F}_q(x))$ given by structure constants
having numerators and denominators of degree at most $d_C\geq 1$.
A  maximal $\F_q[x]$-order $\Lambda $ can be
constructed by an f-algorithm running in time $(n+d_C+\log q)^{O(1)}$.
The output of the algorithm is an $\F_q[x]$-basis for $\Lambda$
whose elements are linear combinations in the original basis of $\A$
with coefficients which are ratios of polynomials of degree at most
$(2n^8+n^6+2n^2)d_C$.
\end{theorem}

Notice that $\sum_{j=0}^d\alpha_jx^j=
x^d\sum_{j=0}^d\alpha_{d-j}\frac{1}{x^j}$.
Therefore a fraction of two polynomials in $x$
of degree at most $d$ can also be written as
a fraction of two polynomials in $\frac{1}{x}$
also of degree at most $d$. Therefore Theorem~\ref{thm:maxord-construct}
gives the following.

\begin{corollary}
\label{cor:revord-construct}
Let $\A$ and $d_C$ be as in Theorem~\ref{thm:maxord-construct}.
Then a  maximal $\F_q[\frac{1}{x}]$-order $\Delta $ can be
constructed by an f-algorithm running in time $(n+d_C+\log q)^{O(1)}$.
The output of the algorithm is an $\F_q[\frac{1}{x}]$-basis for $\Delta$
whose elements are linear combinations in the original basis of $\A$
with coefficients which are ratios of polynomials (in $x$) of degree at most
$(2n^8+n^6+2n^2)d_C$.
\end{corollary}

We remark that we will actually need an
$\F_q[\frac{1}{x}]_{(\frac{1}{x})}$-basis
for a maximal $\F_q[\frac{1}{x}]_{(\frac{1}{x})}$-order.
Obviously, for this an $\F_q[\frac{1}{x}]$-basis for an
$\F_q[\frac{1}{x}]$-order $\Delta$ whose localization
at the the prime $\frac{1}{x}$ is maximal, will do. Therefore it will be
actually sufficient to apply the main steps of the order increasing
algorithm only for the prime $\frac{1}{x}$ of $\F_q[\frac 1x]$.

\section{Finding a rank 1 idempotent in $\A$}
\label{sec:findidemp}

Let $R\subseteq \F_q(x)$ be the set of rational functions
having degree at most 0 (note that the 0 polynomial has degree
$-\infty$ hence it also belongs to $R$). Thus, if $f,g \in\F_q[x]$,
$g\not= 0$, then $\frac{f}{g}\in R$ iff $\deg f\leq \deg g$.
It is easy to see that $R$ is a subring of $\F_q(x)$.
Actually $R$ is the valuation ring for the valuation $-\deg$ of $\F_q(x)$.
An alternative view is that $R=\F_q[\frac{1}{x}]_{(\frac{1}{x})}$, the localization
of the ring $\F_q[\frac{1}{x}]$ at the prime ideal
$(\frac{1}{x})$. (In fact, one readily verifies that the elements of $R$ are
precisely the functions of the form $f(\frac{1}{x})/g(\frac{1}{x})$, where $f,g$ are
univariate polynomials over $\F_q$ and the constant term of $g$
is not 0.) Thus $R$ is a discrete valuation ring, and as such,
a principal ideal ring.

The main structural result of the paper is the following theorem.
It identifies a finite subalgebra $C$ of modest size in ${\cal A}$, which contains
a primitive idempotent of ${\cal A}$.

\begin{theorem} \label{thm:mainstruct}
Let $\A\cong M_n(\F_q(x))$ and let $\Lambda$ be
a maximal $\F_q[x]$-order in $\A$.
Also, let $R$ be the subring of $\F_q(x)$ discussed above,
that is, the set of rational functions of degree
at most zero. Let $\Delta$ be a maximal $R$-order
in $\A$. Let $b_1,\ldots,b_{n^2}$
be an $\F_q[x]$-basis
of $\Lambda$, and for $j=1,\ldots,{n^2}$
let $d_j$ be the smallest integer
such that $\frac{1}{x^{d_j}}b_j\in\Delta$. Let
$d_{min}=\min\{d_j:1\leq j\leq n^2\}$
$d_{max}=\max\{d_j:1\leq j\leq n^2\}$.
Then
\begin{enumerate}
\item[(i)] For every element $a\in \Lambda\cap \Delta$
we have $a=\sum \alpha_ib_i$, where the $\alpha_i$ are polynomials
in $\F_q[x]$ of degree
at most $n^2d_{max}-d_{min}$.
\item[(ii)] $\Lambda\cap\Delta$ contains a primitive idempotent of $\A$.
\end{enumerate}
\end{theorem}

\begin{proof}
Let $\phi:\A\rightarrow M_n(\F_q(x))$ be an algebra isomorphism such that
$\phi(\Delta)=M_n(R)$. (Such a $\phi$ exists by
Corollary~\ref{cor:max_orders}.)
We show
that the $\F_q[x]$-lattice $\phi(\Lambda)$ in $M_n(\F_q(x))$
(the latter considered as $\F_q(x)^{n^2}$) has determinant $1$.
To see this, let $B$ be the matrix whose columns
form an $\F_q[x]$-basis for the $\F_q[x]$ lattice $\phi(\Lambda)v\subset
\F_q(x)^n$ where $v$ is a nonzero vector from $\F_q(x)^n$.
Then $\phi(\Lambda)=BM_n(\F_q[x])B^{-1}$. The claim
on the determinant follows
from that the standard lattice $\F_q[x]^{n^2}$ has determinant one
and from that the conjugation
$X\mapsto BXB^{-1}$, considered as an $\F_q(x)$-linear transformation
on $\F_q(x)^{n^2}$, has determinant one.
For the latter, notice that
multiplication by $B^{-1}$ from the right is similar to a block diagonal
matrix consisting of $n$ copies of $B^{-1}$, and hence has determinant
$(\det B^{-1})^n$, while multiplication by $B$ from the left
has determinant $(\det B)^n$.

Let $C=\Lambda\cap \Delta$. As $\Delta=\phi^{-1}(M_n(R))$,
$C$ can be characterized as the set of the elements $a$ of $\Lambda$
such that $\phi(a)$ has no entries of positive degree.
As both $\Delta$ and $\Lambda$ are $\F_q$-algebras, so is $C$.

Notice that for $0\neq a\in \A$ the degree of $\phi(a) \in M_n(\F_q (x))$
(the maximum of the degrees of  the entries of the matrix $\phi (a)$) is just the
minimal (possibly negative) integer $r$ such
that $\frac{1}{x^r}\phi(a)\in M_n(R)$, or,
equivalently, $x^{-r}a\in \Delta$.
It follows that the degrees of
the entries of $\phi(b_j)$ are bounded by $d_{max}$ and hence
the orthogonality defect of the basis
$\phi(b_1),\ldots,\phi(b_{n^2})$ for $\phi(\Lambda)$ is at most
$n^2d_{max}$, because $|\det \phi(\Lambda)|=0$. Therefore, for $a=
\sum_{j=1}^{n^2}\alpha_jb_j\in C$
Lemma~\ref{orthogonality defect}
gives that $\alpha_j$ has degree at most $n^2d_{max}-d_{min}$,
showing statement (i).

To establish statement (ii), consider an invertible  matrix $B\in M_n(\F_q(x))$
for which $\phi(\Lambda)=B^{-1}M_n(\F_q[x])B$. Let us consider
the lattice $L_1=B^{-1}\F_q[x]^n$ in $\F_q(x)^n$. The determinant
of $L_1$ is obviously $\det B^{-1}$. Let us denote by $\delta$ be the
degree of $\det B$. Let $B^{-1}u_1,\ldots,B^{-1}u_n$, with $u_i \in
\F_q[x]^n$, be an
$\F_q[x]$-basis of orthogonality defect zero
for $L_1$. One can obtain such a basis by lattice basis reduction.
Similarly, let $L_2=B^T\F_q[x]$. Then $L_2$ is an
$\F_q[x]$-lattice having determinant $\det B$.
Let $B^Tu_1',\ldots,B^Tu_n'$, with
$u_i' \in
\F_q[x]^n$, be a basis of defect zero
for $L_2$. Now we define a graph. We connect $u_i$ with $u_j'$ with an edge
if ${u_j'}^Tu_i\neq 0$. This defines a bipartite graph having
these $2n$ vectors as vertices satisfying Hall's criterion
for having a perfect matching. (A set of $s$ vectors from
$u_1,\ldots,u_n$ having less than $s$ neighbors would span
a subspace of dimension $s$ having an orthocomplement having
dimension larger than $n-s$.) By changing the order of
$u_j'$s we arrange that ${u_i'}^Tu_i\neq 0$ ($i=1,\ldots,n$).
We have
$$\sum_{j=1}^n(|B^{-1}u_j|+|B^T u_j'|)=\sum_{j=1}^n|B^{-1}
u_j|+\sum_{j=1}^n|B^T
u_{j}'|=-\delta+\delta=0,$$
 whence
there exists at least one index $i$, such that
the maximum degree of the coordinates of $B^{-1}u_i$
and the maximum degree of the coordinates of $B^Tu_i'$
add up to at most zero. Let $i$ be such an index and let $S$ resp.~$S'$
be the matrix whose first column is $u_i$ resp.~$u_i'$, and  whose
remaining entries are zero. Now $Z=B^{-1}S{S'}^TB$ is a
matrix whose entries are of degree at most zero. Also,
$Z\in \phi (\Lambda)$. Therefore $\phi^{-1}(Z)$ is in $C$.
Furthermore,
$Z$ has rank one as it is similar to $S{S'}^T=u_i{u_i'}^T$. Also,
as $(u_i{u_i'}^T)^2=\mu u_i{u_i'}^T$ where $\mu={u_i'}^Tu_i\neq 0$.
It follows that the minimal polynomial of $Z$ over $\F_q(x)$
as well as that of $\phi^{-1}(Z)$ is $X^2-\mu X$
with a nonzero $\mu\in \F_q(x)$. As $\phi^{-1}(Z)\in \Lambda\cap \Delta$,
we have $\mu\in \F_q[x]\cap R=\F_q$. Now $e=\frac{1}{\mu}\phi^{-1}(Z)$
is an idempotent in $C$ such that $\phi(e)$ has rank one.
\end{proof}
\begin{remark}
We give an example of a $C$ which is not isomorphic to a full matrix algebra over $\mathbb{F}_q(x)$. Let $\Lambda=B^{-1}M_2(\mathbb{F}_q[x])B$ where $B$ is the following matrix:
$$\begin{pmatrix}
\frac{1}{t} & 0 \\
0 & t 
\end{pmatrix}.$$
Let $\Gamma=M_2(R)$, i.e. those matrices whose degree is at most 0. Then $C=\Gamma\cap\Lambda$ is generated as an $\mathbb{F}_q$ vector space by the following matrices:
$$\begin{pmatrix}
1 & 0 \\
0 & 0
\end{pmatrix},
\begin{pmatrix}
0 & 0 \\
1 & 0
\end{pmatrix},
\begin{pmatrix}
0 & 0 \\
\frac{1}{t} & 0
\end{pmatrix},
\begin{pmatrix}
0 & 0 \\
\frac{1}{t^2} & 0
\end{pmatrix},
\begin{pmatrix}
0 & 0 \\
0 & 1
\end{pmatrix}.$$
Note that $C$ has dimension 5 over $\mathbb{F}_q$, hence it cannot be isomorphic to $M_2(\mathbb{F}_q)$. As a matter of fact, it is not even semisimple. The radical of $C$ consists of those matrices whose diagonal entries are 0. Finally, note that $C/\Rad~ C\cong \mathbb{F}_q\oplus \mathbb{F}_q$.  
\end{remark}
\medskip

For finding a primitive idempotent of $\A$ inside $C$
we can use the method described in the proof of the
following lemma.

\begin{lemma}
\label{lem:find-idemp}
Let $C$ be the finite $\F_q$-algebra from Theorem~\ref{thm:mainstruct},
and let $e_1,\dots,e_{r}$ be a complete system of orthogonal primitive idempotents in $C$.
Then there exists an $i$ such that $e_i$ is a rank 1 idempotent in $\A$.

Having a basis of $C$ at hand (a subset of $\A$), one can find such an
idempotent by a polynomial time $f$-algorithm.
\end{lemma}

\begin{proof} We note first, that the identity element of $\A$ is in $C$,
hence $C$ has idempotents.
Let $x\in C$ be an element which is a rank 1 idempotent in $\A$.
By Theorem~\ref{thm:mainstruct} such an $x$ exists. Next observe that there exists an
index $i$,  for which $e_ix$ is not in the radical of $C$. For, otherwise
$\sum _{i=1}^r e_ix=x$ would be in the
radical of $C$, which is impossible, as $x$ is not nilpotent.
Let us denote this primitive idempotent $e_i$ by $e$. Since $ex$ is not in the radical
of $C$, the right ideal $exC$ it generates in $C$ contains a nonzero idempotent $f$.
Indeed, we can consider this right ideal as an $\F_q$-algebra which is not
nilpotent.
Hence if we factor out its radical, then we have a nonzero  idempotent there
(\cite[Corollary 2.2.5]{Kirichenko}), which can be
lifted to an idempotent in $exC$ (\cite[Corollary 3.1.2]{Kirichenko}).
Write $f=exy$ with a suitable $y\in C$. We have $ef=e(exy)=e^2xy=exy=f$. We
verify now that both $fe$ and $e-fe$ are idempotent elements:
$$(fe)^2=fefe=f(ef)e=ffe=fe$$
and
$$(e-fe)^2=e^2+(fe)^2-efe-fee=e+fe-fe-fe=e-fe.$$
Furthermore, they are orthogonal:
$$fe(e-fe)=(fee)-(fe)^2=fe-(fe)^2=0$$
and
$$(e-fe)fe=(ef)e-(fe)^2=fe-(fe)^2=0.$$
Since $e$ is a primitive idempotent, one has either $fe=0$ or $fe=e$.
We show that the first case cannot happen. If $fe=0$ then $fef=0$.
However, $fef=f^2=f$ which is not zero. This implies that $fe=e$, and
$e=exye$. Since $x$ had rank 1 in $\A$, $e$ also has rank 1 in $\A$.

\smallskip

As for the computational part of the statement, first one computes a
Wedderburn-Malcev complement in $C$: a subalgebra $B$ of $C$ which is isomorphic to
$C/Rad(C)$. This can be done in deterministic polynomial time using the
algorithm of
\cite[Theorem 3.1]{Kuronya}. Then we can use for example the polynomial time
$f$-algorithms of \cite{FR} and \cite{Ro1} to compute a complete
system of primitive idempotents in $B$. To calculate
ranks, we can use the fact that for $a\in \A$ the left ideal $a\A$
has dimension $rn$ over $\F _q(x)$ where $r$ is the rank of
$a$ (considered as an
$n$ by $n$ matrix).
\end{proof}

\smallskip

We prove a bound on $d_{min}$ and $d_{max}$ in the case when
$\Lambda$ and $\Delta$ are the maximal orders constructed in
Theorem \ref{thm:maxord-construct} and Corollary \ref{cor:revord-construct},
respectively.
$\Lambda$ is an $\F_q[x]$-order and
$\Delta$ is viewed as an $R$-order here.

\begin{lemma}
\label{lem:dmax}
For the pair of maximal orders as above, we have $d_{max}\leq
(2n^8+2n^6+2n^2)d_C$ and
$d_{min}\geq -2 (2n^8+n^6+2n^2)d_C$
\end{lemma}

\begin{proof} For short, we write $L=(2n^8+n^6+2n^2)d_C$.
Let $a_1,\dots, a_{n^2}$ be the input basis of  $\A$ we use
in the algorithms of
Theorem \ref{thm:maxord-construct} and Corollary \ref{cor:revord-construct}.
We know that the numerators and denominators of the structure constants
for $\A$ are polynomials of degree at most $d_C$.  Let $g^*(1/x)$ be the
smallest
common denominator of the structure constants when written as rational
functions in $\frac 1x$. The degree of $g^*$ is at most $n^6d_C$.
We know that the $g^*(1/x)a_i$ are in the starting almost $\F_q [\frac 1x]$-order
$\Delta_0$, hence they are also in $\Delta$.
Also, one can then write
$$ g^*\left( \frac 1x \right)= \frac{1}{x^\ell}h\left( \frac 1x \right) $$
where
$h(y)\in \F _q[y]$ and $h(0)\not= 0$.  We have here $\ell\leq n^6d_C$.
We claim that $\frac{1}{x^{n^6d_C}}a_i\in \Delta $ hold for every $i$.
Indeed
$$ \frac{1}{x^{n^6d_C}}a_i= \frac{\frac{1}{x^{n^6d_C-\ell}}}{h(\frac 1x)}
\cdot
g^*\left( \frac 1x\right) a_i.$$
Here the first factor is in $R$, the second is in $\Delta$, thus giving the
claim.

We know from Theorem  \ref{thm:maxord-construct}
that every basis element $b_j$ of $\Lambda$ is a linear combination of the
$a_i$ with coefficients $\alpha _i\in \F_q(x)$, and the numerator as well as
the denominator of $\alpha _i$ has degree  at most $L$. We claim now
that  $\frac{1}{x^{n^6d_C+ L}}b_j\in \Delta$. Indeed, we have
$$ \frac{1}{x^{n^6d_C +L}}\alpha_i a_i=
\left( \frac{1}{x^{n^6d_C}}a_i \right) \cdot
\left( \frac{1}{x^{ L}}\alpha_i\right).$$
The first factor is in $\Delta$, the second is in $R$ and the upper bound
follows.

\smallskip

As for $d_{min}$, we observe that the coefficients for the elements of
$\Lambda$ in the basis $\{a_i\}$ are rational functions of degree at
least $-L$ (Theorem \ref{thm:maxord-construct}).
Similarly, by Corollary \ref{cor:revord-construct}
the coefficients for the elements of $\Delta$
in the basis $\{a_i\}$ are rational
functions of degree at most $L$. It follows that for $d< -2L$ the element
$ \frac{1}{x^d}b_j$ can not be in $\Delta$, as the coefficient
$ \frac{1}{x^d}\alpha _i$ has degree at least $L+1$.
\end{proof}

\medskip

Now we turn to the algorithmic task of finding (an  $\F_q$-basis of) $C$.

\begin{lemma}
\label{lem:find-finite-alg}
Let $b_1,\ldots,b_{n^2}$ be the  $\F_q[x]$-basis
of  $\Lambda$ constructed by the algorithm of Theorem
\ref{thm:maxord-construct}, and let
$u_1,\ldots ,u_{n^2}$ be the $R$-basis of $\Delta$ constructed by the method
of Corollary \ref{cor:revord-construct}.
From these data we can construct an $F_q$-basis of $C$ in deterministic
polynomial time.
\end{lemma}

\begin{proof}
We consider the elements of $\A$ as  vectors in the basis
$u_1,\ldots ,u_{n^2}$. This way the elements   of $ \A$ can be viewed as
vectors from $\mathbb{F}_q(x)^{n^2}$ in the usual way:
an element $a\in\A$ with  $a=\sum_{j=1}^{n^2}\alpha_ju_j$
is represented by the vector
$$ (\alpha_1,\dots,\alpha_{n^2})^T\in \mathbb{F}_q(x)^{n^2}.$$
Observe, that a vector as above represents an element of $\Delta$
iff $|\alpha _i|\leq 0$ holds for every $i$.
Consider now the vectors $b_i'   \in \mathbb{F}_q(x)^{n^2} $ representing
the basis elements $b_i$ of
$\Lambda$. They generate a full $\F_q [x]$-lattice (corresponding to
$\Lambda$)  in
$\mathbb{F}_q(x)^{n^2}$.
We next compute a reduced basis $c_1,\dots, c_{n^2}$ of this lattice.
An element
$$ a=\sum_{i=1}^{n^2}\beta _ic_i  ~~~ {\rm with}~ \beta_i \in \F_q[x]
~{\rm for}~i=1,\ldots, n^2 $$
represents an element of $C=\Lambda\cap\Delta $ iff $|a|\leq 0$. We
claim that this latter condition is
equivalent to the set of inequalities
$$ |\beta _i c_i|=|\beta_i|+|c_i|\leq 0, ~~~ i=1,\ldots , n^2.$$
Indeed, as the $\{c_i\}$ is a reduced $\F_q[x]$-basis, from
Lemma \ref{orthogonality defect} we obtain that
\begin{equation}
|\beta_i|\leq |a|+ OD(c_1,\dots,c_{n^2})-|c_i|=|a|-|c_i|
\end{equation}
for every $i$, hence if $|a|\leq 0$ then $|\beta _ic_i|\leq 0$ for every $i$.
Conversely,  $|\beta _ic_i|\leq 0$ for every $i$ obviously
implies that $|a|\leq 0$. We conclude that the elements $x^j c_i$ such that
$1\leq i\leq n^2 $ and $j$ is a natural number  with $j+|c_i|\leq 0$
form an $\F_q$-basis of $C$. Theorem~\ref{thm:mainstruct} and
Lemma~\ref{lem:dmax} provide
a polynomial upper bound for the dimension of $C$ over
$\F_q$, and hence on the number of such elements $x^jc_i$.\footnote{A
polynomial bound for the dimension of $C$ follows also
simply from the polynomiality of the algorithm described here.}

The algorithmic subtasks involved here: change of basis from the input basis to
the basis $\{u_i\}$, and the lattice basis reduction both can be done in
deterministic polynomial time, hence  from $\Lambda$ and $\Delta$ we obtain
$C$ in polynomial time.
\end{proof}

\medskip

The main steps of our algorithm for finding a rank 1 idempotent element
$e\in \A$ are as follows.

\hrulefill

\begin{enumerate}

\item Construct a maximal $\mathbb{F}_q[x]$-order $\Lambda$ and a
maximal $R$-order $\Delta$, by the f-polynomial time algorithms of
Theorem \ref{thm:maxord-construct},
and Corollary \ref{cor:revord-construct}, respectively.

\item Compute an $\mathbb{F}_q$-basis of the finite algebra
$C=\Lambda\cap\Delta$ using the polynomial time algorithm
of Lemma~\ref{lem:find-finite-alg}.

\item With the polynomial time f-algorithm of Lemma~\ref{lem:find-idemp} find a
complete system $e_1,\ldots , e_r$ of orthogonal
primitive idempotents in $C$, and then select an $e_i$ among them which
has rank 1 in $\A$. Finally output this element $e=e_i$.
\end{enumerate}

\hrulefill

\medskip

\noindent
{\em Proof of Theorem~\ref{thm:Main}.}
The correctness and the timing for the first Step follows immediately from
Theorem~\ref{thm:maxord-construct}, and
Corollary~\ref{cor:revord-construct}.
These, and Lemma~\ref{lem:dmax} imply that $C$
admits polynomial size description. Then Lemma \ref{lem:find-finite-alg}
settles Step 2. Correctness and polynomiality for the last Step is
provided by Lemma~\ref{lem:find-idemp}.
$\Box$

\medskip

\paragraph*{Acknowledgement}
Research supported by the Hungarian National Research, Development and Innovation Office - NKFIH, Grants NK105645 and K115288. The authors are grateful to an anonymous referee for helpful remarks and suggestions.


\begin{thebibliography}{MM}


\bibitem{Ajtai}
 M. Ajtai: The shortest vector problem in $L_2$ is NP-hard for
randomized reductions, Proceedings of the 30th annual ACM symposium on
Theory of computing (1998), Dallas, Texas, United States, ACM. pp. 10-19.

\bibitem{Berlekamp}
E.R. Berlekamp:
Factoring polynomials over finite fields,
Bell System Technical Journal 46 (1967), pp. 1853-1859.

\bibitem{CZ}
D.G. Cantor, H. Zassenhaus:
A new algorithm for  factoring polynomials
over finite fields,
Mathematics of Computation 36 (1981), pp. 587-592.

\bibitem{Cremona1}
Explicit $n$-descent on elliptic curves I. Algebra,
Journal f\"ur die reine und angewandte Mathematik, Vol. 615 (2008),pp.  121-155.

\bibitem{Cremona2}
Explicit $n$-descent on elliptic curves II. Geometry,
Journal f\"ur die reine und angewandte Mathematik 632 (2009), pp. 63--84.

\bibitem{Cremona3}
Explicit $n$-descent on elliptic curves III. Algorithms,
Mathematics of Computation 84 No.292 (2015), 895-922.

\bibitem{Kirichenko}
Y. Drozd, V.V. Kirichenko: Finite dimensional algebras,
Vyshcha Shkola, Kiev, 1980.

\bibitem{FR}
K. Friedl, L. R\'onyai: Polynomial time solutions of some problems in computational algebra,
Proceedings of the 17th annual ACM symposium on Theory of computing (1985), Providence, Rhode Island, United States, ACM. pp. 153-162.

\bibitem{GZ03}
M. Giesbrecht, Y. Zhang:
Factoring and decomposing Ore polynomials over $\F_q(T)$,
Proceedingss of the 2003 International Symposium on Symbolic and Algebraic
Computation (ISSAC2003), New York, NY, United States, ACM. pp. 127-134.

\bibitem{GTLN}
J. G\'omez-Torrecillas, F. J. Lobillo, G. Navarro:
Factoring Ore polynomials over $\F_q(t)$ is difficult,
(2015) Preprint arXiv:1505.07252.

\bibitem{Kuronya}
W.A. de Graaf, G. Ivanyos, A. K\"uronya, L. R\'onyai: Computing Levi decompositions, Applicable Algebra in Engineering, Communication and Computing 8 (1997),
pp. 291-304.

\bibitem{Ithesis}
G. Ivanyos: Algorithms for algebras over global field, Ph. D. thesis, Hungarian Academy of Sciences (1996), \url{http://real-d.mtak.hu/261/1/Ivanyos_Gabor.pdf}.

\bibitem{IKRS}
G. Ivanyos, M. Karpinski, L. R\'onyai, N. Saxena:
Trading GRH for algebra: algorithms for factoring polynomials and related
  structures,
Matematics of Computation 81 (2012), pp. 493-531.



\bibitem{ILR}
G. Ivanyos, \'A. Lelkes, L. R\'onyai: Improved algorithms
for splitting full matrix algebras, JP Journal of Algebra, Number
Theory and Applications 28 (2013),
pp. 141-156.



\bibitem{RLIR} G. Ivanyos, L. R\'onyai: On the complexity of finding maximal orders in semisimple algebras over $\Q$, Comput. complexity 3 (1993),
pp. 245-261.

\bibitem{IRS}
G. Ivanyos, L. R\'onyai, J. Schicho: Splitting full matrix algebras over algebraic number fields, Journal of Algebra 354 (2012),
pp. 211-223.

\bibitem{IRSz}
G. Ivanyos, L. R\'onyai, \'A. Sz\'ant\'o: Decomposition of algebras over $\mathbb{F}_q(x_1,... ,x_m)$, Applicable Algebra in Engineering, Communication and Computing 5 (1994),
pp. 71-90.

\bibitem {Lenstra}
A.K. Lenstra: Factoring multivariate polynomials over finite fields, Journal of Computer and System Sciences 30 (2) (1985),
pp. 235-248.

\bibitem{Paulus}
S. Paulus: Lattice basis reduction in function fields,
J. Buhler (Ed.), Proceedings of the Third Symposium on Algorithmic Number
Theory, Portland, Oregon, United States: ANTS-III, Springer LNCS 1423 (1998), pp.
567-575.

\bibitem{Reiner} I. Reiner: Maximal orders, Academic Press, 1975.

\bibitem{Ro1}
L. R\'onyai: Computing the structure of finite algebras, Journal of Symbolic Computation 9 (1990),
pp. 355-373.

\bibitem{Vigneras}
M-F. Vign\'eras: Arithm\'etique des Alg\`{e}bres de Quaternions,
Springer, LNM 800, 1980.



\end{thebibliography}
\end{document}